\documentclass[9pt,technote]{IEEEtran}

\usepackage{graphicx}
\usepackage{amsmath}
\usepackage{amssymb}
\usepackage{amsthm}

\newtheorem{theorem}{Theorem}

\newcommand{\Per}{P}
\newcommand{\Cry}{C}

\newcommand{\setU}{U}

\newcommand{\setY}{Y}
\newcommand{\setZ}{Z}
\newcommand{\setX}{X}

\newcommand{\setStateTrans}{\mathcal{R}}

\newcommand{\setInputTrans}{\mathcal{P}}

\newcommand{\setInputs}{\mathcal{U}}
\newcommand{\setTransmissibleU}{\mathcal{U}_{TI}}
\newcommand{\transmissibleU}{u_{TI}}

\newcommand{\const}{\operatorname{const}}

\usepackage{xifthen}
\usepackage{color}
\usepackage[normalem]{ulem}
\usepackage{xspace}

\allowdisplaybreaks

\begin{document}
\title{The Mechanism of Scale-Invariance}

\author{\IEEEauthorblockN{Moritz Lang}\\
\IEEEauthorblockA{\textit{University of Applied Sciences Technikum Wien} \\
H\"ochst\"adtplatz 6, 1200 Wien, Austria \\
moritz.lang@technikum-wien.at}}                                                
\maketitle
\begin{abstract}                          
A system is invariant with respect to an input transformation if we can transform any dynamic input by this function and obtain the same output dynamics after adjusting the initial conditions appropriately. Often, the set of all such input transformations forms a Lie group, the most prominent examples being scale-invariant ($u\mapsto e^pu$, $p\in\mathbb{R}$) and translational-invariant ($u\mapsto pu$) systems, the latter comprising linear systems with transfer function zeros at the origin.
Here, we derive a necessary and sufficient normal form for invariant systems and, by analyzing this normal form, provide a complete characterization of the mechanism by which invariance can be achieved. In this normal form, all invariant systems (i) estimate the applied input transformation by means of an integral feedback, and (ii) then apply the inverse of this estimate to the input before processing it in any other way. 
We demonstrate our results based on three examples: a scale-invariant ``feed-forward loop'', a bistable switch, and a system resembling the core of the mammalian circadian network.
\end{abstract}


\section{Introduction}
The concept of invariances with respect to input transformations originated in the life sciences \cite{Goentoro2009}. It was introduced to formalize the observation that several naturally evolved biological networks, like bacterial chemotaxis \cite{Lazova2011}, show (approximately) the same response dynamics when excited by scaled environmental signals after equilibration to accordingly scaled, constant environments (Figure~\ref{fig:firstexamples}A,C\&E).
In its original formulation \cite{Goentoro2009,Shoval2010}, the concept was restricted to the special case of multiplicative input transformations ($u\mapsto s u$, $s > 0$), and is then referred to as scale-invariance (SI) or fold-change detection (FCD). It was furthermore initially required that every scale-invariant system possesses a globally asymptotically stable steady-state.
In subsequent work, the concept was however quickly generalized to (i) also cover other types of input transformations like translations or reflections \cite{Shoval2011}; and (ii) to not pose any restrictions on the number or stability of steady-states, such that the concept can nowadays as well be applied to e.g. oscillatory or bistable systems \cite{Lang2017} (Figure~\ref{fig:firstexamples}). 

The theoretical basis for the analysis of invariant systems was provided in \cite{Shoval2010}, where the concept of equivariance was introduced and where it was shown that, under mild assumptions, equivariance implies invariance and vice-versa. In \cite{Lang2016}, it was observed that, for input transformations forming one-parametric Lie groups, invariant systems ``calculate'' nonlinear derivatives of the input dynamics and thus generalize the notion of differentiators from linear systems theory. This proposed a close relationship between invariant systems and the (linear) concept of transfer function zeros, which was formally derived in \cite{Lang2017}. In the latter article, also higher-order invariances were introduced, as well as invariances with respect to time-dependent transformations, corresponding to transfer function zeros with higher multiplicity and zeros not lying at the origin, respectively. Together, these results established the concept of input invariances as a general dynamic property of nonlinear systems closely related and complementing zero dynamics \cite[p. 162ff]{Isidori1995}.

\begin{figure*}[tb]
	\centering
		\includegraphics[width=0.95\textwidth]{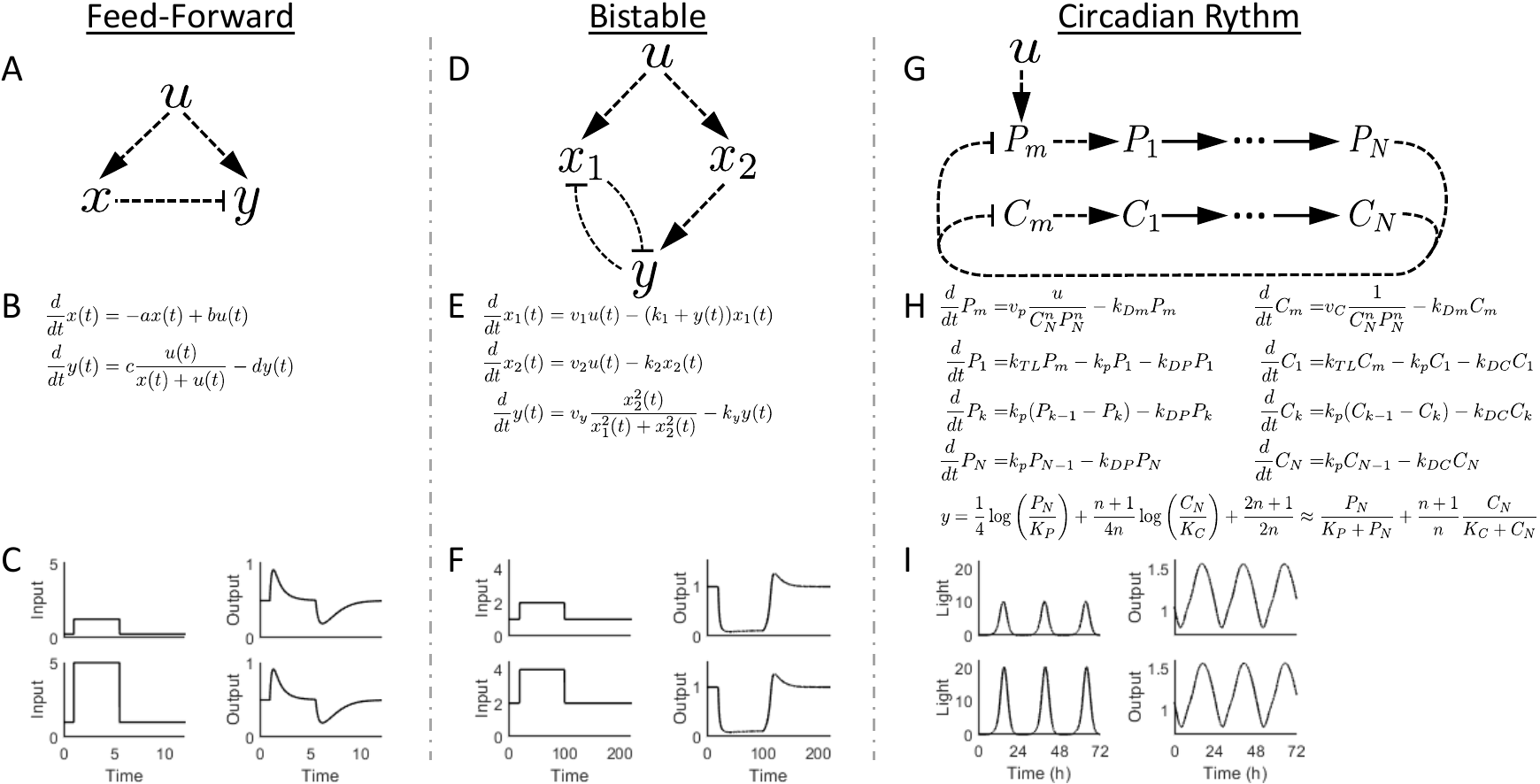}
	\caption[Example]{\label{fig:firstexamples}
	Three examples of scale-invariant networks. Left: incoherent feed-forward loop \cite{Lang2016}; center: bistable network; and right: circadian rhythm. Top: network diagrams of the systems, where dotted arrows with tips and bars represent activation and inhibition, respectively, and solid arrows conversions. Middle: ODEs of the three systems. Bottom: input/output dynamics for scaled versions of the same input dynamics.
	Parameters feed-forward: $a=1$, $b=4$, $c=10$, $d=4$;
	parameters bistable: $v_1=1$, $v_2=0.1$, $v_y=3.5$, $k_1=0.15$, $k_2=0.1$, $k_y=2$; 
	parameters circadian rhythm: $N = 4$, $k_{TL}=k_{DC}=k_{DP} = 0.25h^{-1}$, $k_{Dm} = 0.5h^{-1}$, $v_C = v_P = 0.01h^{-1}$, $k_p = 0.5h^{-1}$, $n=2$.
	}
\end{figure*}

Despite these advances, it yet remains unclear by which mechanisms invariance can be achieved. Similarly, formal methods to design input invariant systems showing desired response dynamics are still lacking. 
In this article, we fill this gap by deriving a (necessary and sufficient) normal form for systems invariant with respect to Lie groups of input transformations, thus covering the practically relevant cases of scale- and translational invariant systems. This normal form concisely explains how invariance can achieved by showing that all invariant systems rely on variations of the same, underlying mechanism. Furthermore, our article also lies the foundation for the rational design of invariant systems, which can be done by simply choosing the system-dependent terms of the normal form appropriately. 

\section{Background and Notation}\label{sec:scaleInvariance}
The exposition in this section shortly summarizes the theoretical foundation described in \cite{Shoval2010,Shoval2011,Lang2017} and closely follows the notation established in these articles.

We consider systems of ordinary differential equations with single inputs and outputs (systems in the following) of the form
\begin{subequations}\label{systemEq}
\begin{align}
\frac{d}{dt}x(t)&=f(x(t),u(t)),\ x(0)=x_0\\
y(t)&=h(x(t)).
\end{align}
\end{subequations}
The states $x(t)=[x_1(t),\ldots,x_n(t)]^T\in\setX$, $n\geq 1$, the input $u(t)\in\setU$ and the output $y(t)\in\setY$ depend on the time $t\in\mathbb{R}_{\geq 0}$, and take values in some state space $\setX\subseteq\mathbb{R}^n$, input space $\setU\subseteq\mathbb{R}$ and output space $\setY\subseteq\mathbb{R}$, respectively. The dynamics of the system are described by the vector field $f:\setX\times\setU\rightarrow\mathbb{R}^n$, while $x_0\in\setX$ denotes the initial conditions at $t=0$, and $h:\setX\rightarrow\setY$ a function mapping the state of the system to the output.
We assume that $f$ and $h$ are analytic, and that all admissible inputs $u\in\setInputs$, $u:\mathbb{R}_{\geq 0}\rightarrow\setU$ are piecewise continuous, i.e. $\setInputs\subseteq\mathcal{PC}(\mathbb{R}_{\geq 0}, \setU)$. We further assume that, for every initial condition $x_0\in\setX$ and every input $u\in\setInputs$, there exists a unique, piecewise differentiable and continuous solution for all $t\geq 0$ denoted by $\xi(t,x_0,u):=x(t)$. 

For the system (\ref{systemEq}), we consider sets $\setInputTrans=\{\pi_p\}_{p\in P}$ of one-to-one onto input transformations $\pi_p:U\rightarrow U$ parametrized by $p\in P$ in the interval $P\subseteq\mathbb{R}$. By a slight abuse of notation, we also denote by $\pi_p:\setInputs\rightarrow\setInputs$ the corresponding (point-wise) transformations of input trajectories defined by $\pi_p(u)(t)=\pi_p(u(t))$.
We assume that $\setInputTrans$ forms a one-parameter Lie group under function composition $\circ$ with law of composition given by $\phi:P\times P \rightarrow P$, i.e. $\pi_{p_2}\circ \pi_{p_1}=\pi_{\phi(p_1,p_2)}$. Recall that this implies that $\pi_p$ is differentiable in $\setU$ and analytic in $P$, and that $\phi$ is analytic in both its parameters \cite[p. 34]{Bluman1989}. 
In the following, we assume that $\setInputTrans$ is parametrized additively, i.e. such that $P=\mathbb{R}$, $\phi(p_1,p_2)=p_1+p_2$, and $\pi_0$ is the identity transformation. By the first fundamental theorem of Lie \cite[p. 37]{Bluman1989}, such an additive parametrization is always possible. For example, translations of the input are described by $\pi_p(\bar{u})=\bar{u}+p$ and scalings by $\pi_p(\bar{u})=e^p\bar{u}$, with $\bar{u}\in\setU$.

Given these definitions, the system (\ref{systemEq}) is invariant with respect to $\setInputTrans$ if, for every initial condition $x_0\in\setX$, every input $u\in\setInputs$ and every $\pi_p\in\setInputTrans$, there exists a $x_0'\in\setX$ such that \cite{Lang2017}
\begin{equation}
h(\xi(t,x_0,u))=h(\xi(t,x_0',\pi_p(u))).
\end{equation} 
Note that different to earlier definitions of invariance \cite{Shoval2010}, this definition does not require that the system possesses a globally asymptotically stable steady-state.

It was shown that the system (\ref{systemEq}) is invariant with respect to the input transformations $\setInputTrans$ if and only if it is equivariant with respect to the same transformations \cite{Shoval2010,Lang2017}, i.e. if there exist state transformations $\setStateTrans=\{\rho_p:\setX\rightarrow\setX\}_{p\in\mathbb{R}}$ such that, for every $\bar{u}\in\setU$ and every $\bar{x}\in\setX$,
\begin{subequations}\label{equivariance}
\begin{align}
f(\rho_p(\bar{x}),\pi_p(\bar{u}))&=\frac{\partial\rho_p}{\partial \bar{x}}(\bar{x})f(\bar{x},\bar{u})\\
h(\rho_p(\bar{x}))&=h(\bar{x}).
\end{align}
\end{subequations}

\section{The Mechanism of Invariance}
We define the following normal form for systems (\ref{systemEq}) invariant with respect to a one parameter Lie group $\setInputTrans=\{\pi_p\}_{p\in\mathbb{R}}$ of input transformations (Figure~\ref{fig:internalModel}):
\begin{subequations}\label{IM_system}
\begin{alignat}{6}
&& \frac{d}{dt}\hat{p}(t)	&=e(t)&&\label{common1}\\
      \text{\smash{\raisebox{\dimexpr.5\normalbaselineskip+1\jot}{common part}}}
&			\qquad\llap{\smash{\raisebox{\dimexpr.5\normalbaselineskip+1\jot}{$\left\{\begin{array}{c}\null\\[\jot]\null\end{array}\right.$}}}
&\hat{u}(t)							&=\pi_{-\hat{p}(t)}(u(t)))&&\label{common2}\\
&&\frac{d}{dt}z(t)				&=f_z(z(t), \hat{u}(t))&&\label{variable1}\\
&&e(t)											&=h_e(z(t), \hat{u}(t))&&\label{variable2}\\
      \text{\smash{\raisebox{\dimexpr1.\normalbaselineskip+2.0\jot}{variable part}}}
&			\qquad\llap{\smash{\raisebox{\dimexpr1.\normalbaselineskip+2\jot}{$\left\{\begin{array}{c}\null\\[\jot]\null\\[\jot]\null\end{array}\right.$}}}
&y(t)											&=h_z(z(t)).&&\label{variable3}
\end{alignat}
\end{subequations}
This normal form naturally decomposes into a part common to all invariant networks (\ref{common1}--\ref{common2}), and a variable part (\ref{variable1}--\ref{variable3}). The common part has a single state $\hat{p}(t)\in\mathbb{R}$, whereas the variant part has $n-1$ states $z\in\setZ\subseteq\mathbb{R}^{n-1}$. The two parts are interconnected in a feedback loop, whereby the output $\hat{u}$ of the common part serves as the input of the variable part, and the output $e$ of the variable part serves as an input to the common part. Note, that the variable part (i.e. $f_z$, $h_e$ and $h_z$) only indirectly depends on the input $u$ via $\hat{u}$.

For the following theorem, recall that we assume that every Lie group is parametrized additively.
\begin{theorem}\label{theorem:normalForm}
Given the assumptions stated in Section~\ref{sec:scaleInvariance}, the system (\ref{systemEq}) is invariant with respect to a one parameter Lie group $\{\pi_p\}_{p\in\mathbb{R}}$ of input transformations if and only if it can be transformed into normal form (\ref{IM_system}).
Furthermore, given state transformations $\{\rho_p\}_{p\in\mathbb{R}}$ satisfying (\ref{equivariance}), a transformation $(z,\hat{p})=(\delta_z(x),\delta_{p}(x))$ into normal form is given by $n-1$ functionally independent solutions of the linear first-order homogenous partial differential equation $E_\rho\delta_{z,i}(d)=0$, and a particular solution of $E_\rho\delta_{p}(d)=1$, with $E_\rho=\sum_{i}\left.\frac{\partial\rho_{p,i}}{\partial p}(x)\right|_{p=0}\frac{\partial}{\partial x_i}$ the infinitesimal generator of $\rho_p$.
\end{theorem}
\begin{proof}
For sufficiency, consider a system given in normal form (\ref{IM_system}). Then, $\rho_p(z,\hat{p})=(z,\hat{p}+p)$ satisfies (\ref{equivariance}), that is, the system is equivariant and thus invariant with respect to $\{\pi_p\}_{p\in\mathbb{R}}$.

For necessity, consider a system (\ref{systemEq}) invariant with respect to $\{\pi_p\}_{p\in\mathbb{R}}$, with state transformations $\{\rho_p\}_{p\in\mathbb{R}}$ satisfying (\ref{equivariance}). 
Recall that for every one-parameter Lie group $\{\tau_p:D\rightarrow D\}_{p\in\mathbb{P}}$ of transformations, there exists a one-to-one and continuously differentiable transformation $\delta_\tau:D\rightarrow D^*$, $d^*=\delta_\tau(d)$, satisfying $\tau_p^*(d^*)=\tau_p^*(\delta_\tau(d))=d^*+(0,\ldots,0,p)$. The elements $\delta_{\tau,1},\ldots,\delta_{\tau,n-1}$ of $\delta_{\tau}$ are given by $n-1$ functionally independent solutions of the linear first-order homogenous partial differential equation $E_\tau\delta_{\tau,i}(d)=0$, and $\delta_{\tau,n}$ satisfies $E_\tau\delta_{\tau,n}(d)=1$, with $E_\tau=\sum_{i}\left.\frac{\partial}{\partial p}{\tau_{p,i}}(d)\right|_{p=0}\frac{\partial}{\partial d_i}$ \cite[p. 45ff]{Bluman1989}. 
Let $\delta_{\rho}:\setX\rightarrow\setZ\times\mathbb{R}$, $(z,\hat{p})=\delta_{\rho}(x)$, and $\delta_{\pi}:\setU\rightarrow\mathbb{R}$, $u^*=\delta_\pi(u)$, be such transformations for $\{\rho_p\}_{p\in\mathbb{R}}$ and $\{\pi_p\}_{p\in\mathbb{R}}$, respectively.
In $\setZ\times\mathbb{R}$, the Jacobian $\frac{\partial\rho_p^*}{\partial (z,\hat{p})}$ is the identity, and (\ref{equivariance}) becomes
\begin{align*}
f^*(z,\hat{p}+p, u^*+p)&=f^*(z, \hat{p},u^*)\\
h^*(z,\hat{p}+p)&=h^*(z,\hat{p}),
\end{align*}
with $f^*(z, \hat{p},u^*)=\frac{\partial\delta_\rho}{\partial x}(\delta_\rho^{-1}(z, \hat{p}))f(\delta_\rho^{-1}(z, \hat{p}),\delta_\pi^{-1}(u^*))$ and $h^*(z,\hat{p}):=h(\delta_\rho^{-1}(z,\hat{p}))$. 
Since the system is irreducible, this implies that $f^*$ can only depend on $u^*-\hat{p}$, but not on $u^*$ or $\hat{p}$ separately, and $h^*$ must not depend on $\hat{p}$ at all. We then obtain that
\begin{align*}
f^*(z,\hat{p},u^*)&=f^*(z,0,u^*-\hat{p})\\
&=f^*(z,0,\pi_{-\hat{p}}^*(u^*))\\
&=f^*(z,0,\delta_{\pi}(\pi_{-\hat{p}}(u)))\\
h^*(z,\hat{p})&=h^*(z,0).
\end{align*}
With $\hat{u}=\pi_{-\hat{p}}(u)$, $(f_z,h_{e})(z,\hat{u})=f^*(z,0,\delta_{\pi}(\hat{u}))$ and $h_z(z)=h^*(z,0)$, we then obtain normal form~(\ref{IM_system}). 
\end{proof}

Given an invariant system in normal form (\ref{IM_system}), we refer to $\hat{p}$ and $\hat{u}$ as its estimates of the input transformation and the ``untransformed input'', respectively, and to $e$ as the adaptation error. This nomenclature is motivated by the following considerations: assume that we transform a signal $u_{org}$ by $\pi_p\in\setInputTrans$ before using it as the input for (\ref{IM_system}), i.e. we set $u=\pi_p(u_{org})$. Further, assume that this input $u$ and the initial conditions $(z_0,\hat{p}_0)$ solve the problem of zeroing the error signal $e(t)$ (see \cite[p. 162ff]{Isidori1995}), i.e. that $e(t)=0$ and, thus, $\hat{p}(t)=\hat{p}_0$ for all $t\geq 0$. Then,
\begin{align*}
\hat{u}(t)=\pi_{-\hat{p}(t)}(\pi_p(u_{org}(t)))=\pi_{p-\hat{p}_0}(u_{org}(t)).
\end{align*}
As we show below, it is always possible to choose the normal form such that $\hat{p}_0=p$. Then, $\pi_{p-\hat{p}_0}=\pi_0$ becomes the identity transformation, and the system's estimate of the untransformed input is precise, i.e. $\hat{u}=u_{org}$.
 
Given a system in normal form (\ref{IM_system}), we refer to all inputs $\transmissibleU\in\setInputs$ for which a $z_0\in\setZ$ exists such that $u$ zeroes the adaptation error $e$ for the initial conditions $z(0)=z_0$ and $\hat{p}(0)=0$, as transmissible inputs $\setTransmissibleU$. 
If, for a given transmissible input $\transmissibleU\in\setTransmissibleU$, $\hat{p}(t)\rightarrow 0$ for all initial conditions in an open region $B\subseteq\setZ$ around $z(0)=z_0$ and $\hat{p}(0)=0$, we say that $\transmissibleU$ is stable transmissible, and if $B=\setZ$, that $\transmissibleU$ is globally stable transmissible. Clearly, if $u_{org}\in\setTransmissibleU$ is (globally) stable transmissible and $u(t)=\pi_p(u_{org}(t))$, we get that $\hat{p}(t)\rightarrow p$ and thus $\hat{u}(t)\rightarrow u_{org}(t)$ for all $(z(0),\hat{p}(0)-p)\in B$.
We thus interpret the normal form (\ref{IM_system}) in the sense that the invariant system estimates the input transformation $\pi_p$ by means of an integral feedback, and immediately applies the inverse $\pi_{\hat{p}}^{-1}$ of this estimated transformation to the input before processing it in any other way (Figure~\ref{fig:internalModel}). When interpreting the adaptation error $e(t)$ as an additional output, the internal model principal (IMP, see e.g. \cite{Sontag2003}) then suggests that the normal form (\ref{IM_system})--more precisely, its variant part (\ref{variable1}\&\ref{variable2})--should contain an internal model capable to generate the transmissible input.

\begin{figure}[tb]
		\includegraphics[width=0.45\textwidth]{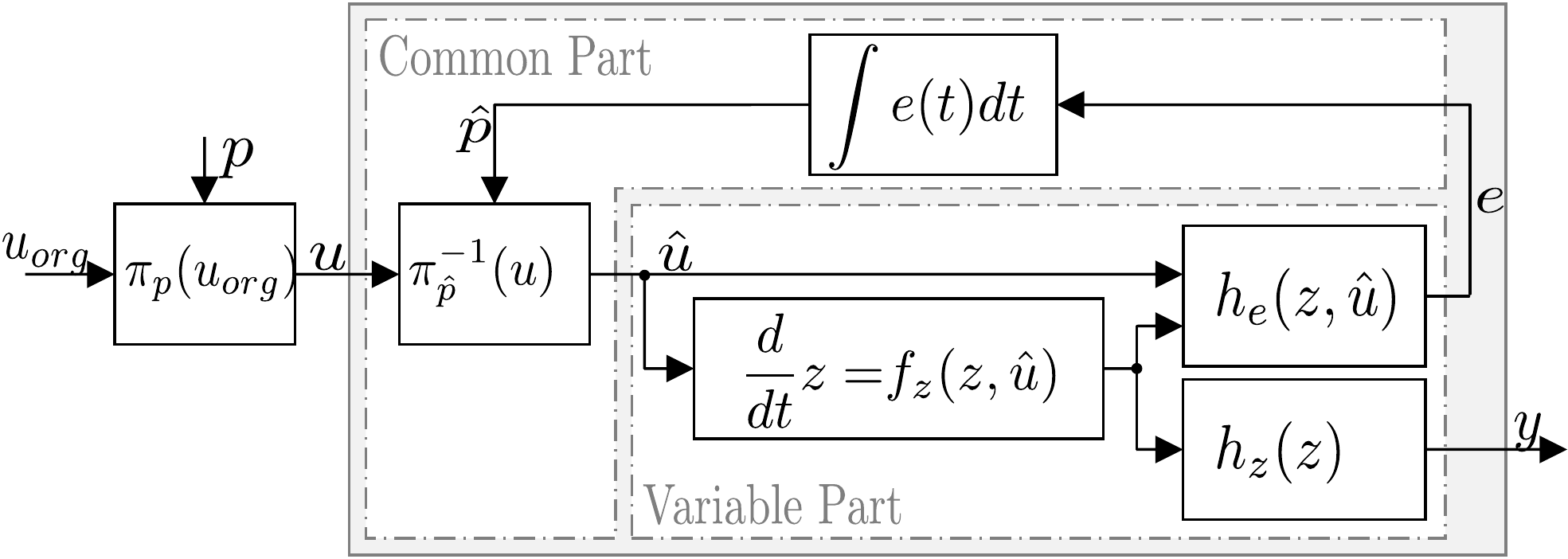}
	\caption[Internal Model]{\label{fig:internalModel}
	Normal form of invariant systems. Every invariant system in normal form (gray box) naturally decomposes into two sub-systems (white boxes): (i) a part common to all invariant systems responsible to estimate and reverse the applied input transformation $\pi_p$ via an integral feedback; and (ii) a system-specific, variable part which shapes the input-output dynamics and determines the transmissible inputs $\setTransmissibleU$.
	}
\end{figure}

The transformation of an invariant system into normal form (\ref{IM_system}) is however not unique: given a transformation $[z,\hat{p}]=[\delta_z(x),\delta_{p}(x)]$ into normal form (\ref{IM_system}), $[\tilde{z},\tilde{p}]=[\tau_z(\delta_z(x)),\delta_{p}(x)+\tau_{p}(\delta_z(x))]$ is a transformation into normal form, too, with $\tau_z:\setZ\rightarrow\tilde{\setZ}$ and $\tau_{p}:\setZ\rightarrow\mathbb{R}$ analytic functions.  
The function $\tau_z$ thereby only changes the coordinates of the variable part of the normal form, but does not influence the dynamics of $\hat{p}$ or $e$. The set $\setTransmissibleU$ of transmissible inputs is thus invariant with respect to $\tau_z$, and we thus w.l.o.g. assume that $\tau_z(z)=z$ in the following. The function $\tau_{p}$, on the other hand, has a profound influence on $\setTransmissibleU$, and thus on our interpretation of the IMP. Given an invariant system in normal form (\ref{IM_system}), this transformation results in (omitting dependencies on $t$ for clarity):
\begin{subequations}
\begin{align*}
\frac{d}{dt}\tilde{p}	=&\tilde{e}\\
\tilde{u}							=&\pi_{-\tilde{p}}(u)\\
\frac{d}{dt}z					%=&\tilde{f}_z(z, \tilde{u})
=&f_z(z,\pi_{\tau_{p}(z)}(\tilde{u}))\\
\tilde{e}							
=&h_e(z, \pi_{\tau_{p}(z)}(\tilde{u}))+\frac{\partial\tau_{p}}{\partial z}f_z(z,\pi_{\tau_{p}(z)}(\tilde{u}))\\
y											=&h_z(z)
\end{align*}
\end{subequations}
Thus, given a system in normal form, if an input $\hat{u}$ zeros the function $h_e(z, \hat{u})+L_{f_z}\tau_{p}$, with $\tau_{p}:\setZ\rightarrow\mathbb{R}$ analytic and $L_{f_z}$ the Lie derivative in the direction of $f_z$, then there exists another normal form of the system with $\pi_{-\tau_{p}(z)}(\hat{u})$ a transmissible input.

If $\tau_{p}$ is constant, this results in a constant offset $\tilde{p}=\hat{p}+\const$ of the invariant system's estimate of the applied input transformation and a corresponding change in the transmissible inputs. This justifies that we can always assume $\hat{p}(t)\rightarrow p$ whenever $\hat{p}(t)$ converges. Note, that already the notion of an original ``untransformed input'' $u_{org}$ which we used above and to which we applied the input transformation $\pi_p$ to generate $u=\pi_p(u_{org})$ was not unique. Indeed, we can always interpret $\tilde{u}_{org}=\pi_q(u_{org})$ as another ``untransformed input'' to which we applied the input transformation $\pi_{p-q}$ to generate exactly the same input $u=\pi_{p-q}(u_{org})$. The result above means that, for every such interpretation, a corresponding normal form exists.

For non-constant transformations $\tau_{p}$ of the normal form, the transmissible inputs and even their number may change more profound. For example, in Section~\ref{section:bistableExample}, we discuss two normal form representations of the same system: (i) the first has a single, globally stable, constant transmissible input, for which the variable part of the normal form is bistable; and (ii) the second has two stable and one unstable constant transmissible inputs, for each of which the variable part of the normal form is monostable. 
In such cases, the IMP suggests that the invariant system possesses (usually distinct) internal models which can generate all sets of transmissible inputs corresponding to all normal forms. However, in many practically relevant cases, the system's structure or prior knowledge likely renders a given normal form ``natural''. Indeed, as our third example shows (Section~\ref{section:circRhythm}), such prior knowledge might exist in form of the transmissible inputs themselves.

\section{Examples}
\subsection{A scale-invariant feed-forward loop}
In the literature, it is often stated that systems can achieve invariance by two distinct mechanisms: negative feedback and incoherent feedforward loops \cite{Adler2018}. In spite of this research, this distinction however seems to exists mainly on the level of specific choices of coordinates used to model the system, however justified, than to be a property of the system itself. 
For example, consider the system in ``feed-forward form'' shown in Figures~\ref{fig:firstexamples}A--C. This system is equivariant with respect to the input and state transformations $\setInputTrans=\{\pi_p(\bar{u})=e^p \bar{u}\}_{p\in\mathbb{R}}$ and $\setStateTrans=\{\rho_p(\bar{x},\bar{y})=[e^p\bar{x},\bar{y}]\}_{p\in\mathbb{R}}$, and thus scale-invariant.

The infinitesimals of the state transformations $\setStateTrans$ are given by $\eta_\rho(x,y)=\left.\frac{\partial}{\partial p}\rho(x,y)\right|_{p=0}=[x,0]^T$, and the infinitesimal generator by $E_\rho(x,y)=\eta_\rho(x,y)\cdot\left[\frac{\partial}{\partial x},\frac{\partial}{\partial y}\right]^T=x\frac{\partial}{\partial x}$. The transformations $z=\delta_{z}(x,y)=y$ and $\hat{p}=\delta_{p}(x,y)=\log(x)$ solve the partial differential equations 
\begin{subequations}
\begin{align}
E_\rho(x,y)\delta_{z}(x,y)&=x\frac{\partial}{\partial x}\delta_{\rho,1}(x,y)=0\\
E_\rho(x,y)\delta_{p}(x,y)&=x\frac{\partial}{\partial x}\delta_{\rho,2}(x,y)=1,\label{ff_PDE2}
\end{align}
\end{subequations}
and transform the system into the normal form
\begin{subequations}
\begin{align*}
\frac{d}{dt}\hat{p}(t)	&=-a+b\hat{u}(t)\\
\hat{u}(t)							&=e^{-\hat{p}(t)}u(t)\\
\frac{d}{dt}z(t)				&=c\frac{\hat{u}(t)}{1+\hat{u}(t)}-d z(t)\\
y(t)											&=z(t).&&
\end{align*}
\end{subequations}

The constant function $\transmissibleU(t)=\frac{a}{b}$ zeros the adaptation error and is thus a transmissible input. For all other constant inputs $u=\pi_p(\transmissibleU)=e^p\transmissibleU$, the system estimates the applied input transformation $\pi_p$ and rescales the input accordingly, i.e. $\hat{u}(t)=\pi_{p-\hat{p}(t)}(\transmissibleU(t))\rightarrow\transmissibleU(t)$.
However, every transformation $\tilde{\delta}_{p}([x,y]^T)=\log(x)+\bar{p}$, $\bar{p}=\const$, is a solution of (\ref{ff_PDE2}), too. Applying this transformation instead of $\delta_{p}$ results in a normal form for which the transmissible input becomes $\tilde{u}_{TI}(t)=\exp(-\bar{p})\frac{a}{b}$.

\subsection{A scale-invariant bistable system}\label{section:bistableExample}
Consider the system shown in Figure~\ref{fig:firstexamples}D--F. The system is equivariant, and thus invariant, with respect to the input transformations $\pi_p(u)=e^pu$, with state transformations given by $\rho_p(x_1,x_2,y)=[e^px_1,e^px_2,y]^T$. 
The infinitesimal generator is given by $E_\rho(x_1,x_2,y)=x_1\frac{\partial}{\partial x_1}+x_2\frac{\partial}{\partial x_2}$, and the corresponding partial differential equations 
are solved by $[z_1,z_2,\hat{p}]^T=[x_1/x_2,y,\log(x_2)]^T$, transforming the system into the normal form
\begin{subequations}\label{bisNormal1}
\begin{align}
\frac{d}{dt}\hat{p}(t)	&=v_2\hat{u}-k_2\\
\hat{u}(t)						&=e^{-\hat{p}(t)}u(t)\\
\frac{d}{dt}z_1(t)				&=(v_1-v_2z_1)\hat{u}+(k_2-k_1-z_2)z_1\\
\frac{d}{dt}z_2(t)				&=v_y\frac{1}{1+z_1^2}-k_yz_2\\
y(t)											&=z_2(t).
\end{align}
\end{subequations}
The normal form's single transmissible input is given by $\transmissibleU(t)=\frac{k_2}{v_2}$. When choosing the parameters appropriately, the variable part is bistabile for this input (\ref{fig:firstexamples}F).

\begin{figure}[tb]
	\centering
	\includegraphics[width=0.4\textwidth]{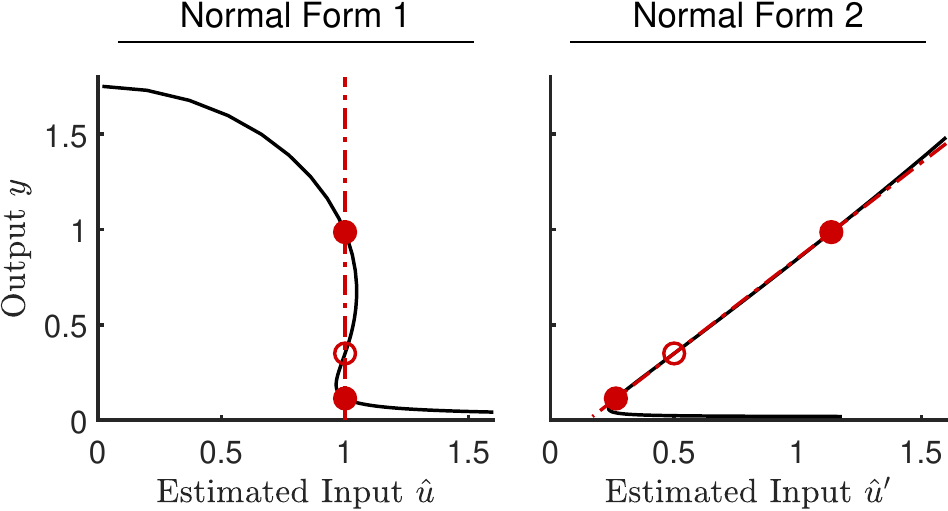}
	\caption{\label{fig:bisNullklines}
	Nullclines of the two normal forms of the scale-invariant bistable system. The normal form (\ref{bisNormal1}) possesses only one transmissible input $\transmissibleU$, for which the variable part of the normal form is bistable (left). In contrast, the normal form (\ref{bisNormal2}) possesses two stable and one unstable transmissible input, for each of which the variable part is mono-stable (right).
	}
\end{figure}

On the other hand, also the transformation $[z_1,z_2,\tilde{p}]^T=[x_1/x_2,y,\log(x_1)]^T$ transforms the system into normal form, now given by
\begin{subequations}\label{bisNormal2}
\begin{align}
\frac{d}{dt}\tilde{p}(t)	&=v_1\tilde{u}-(k_1+z_2)\\
\tilde{u}(t)						&=e^{-\tilde{p}(t)}u(t)\\
\frac{d}{dt}z_1(t)				&=(v_1-v_2z_1)z_1\tilde{u}+(k_2-k_1-z_2)z_1\\
\frac{d}{dt}z_2(t)				&=v_y\frac{1}{1+z_1^2}-k_yz_2\\
y(t)											&=z_2(t).
\end{align}
\end{subequations}
Different to the previous normal form, this normal form possesses three transmissible inputs, two stable and one unstable one (Figure~\ref{fig:bisNullklines}), while the variable part of the system is monostable. It is however easy to see that the bistability was merely shifted from the variable part of the first normal form into the estimation of the untransformed input of the second one--our interpretation changed, while the system's I/O-dynamics remained the same. 

\subsection{Mammalian circadian Rhythm}\label{section:circRhythm}
Consider the model shown in Figure~\ref{fig:firstexamples}G--I. This model is a simplified representation of the core of the mammalian circadian rhythm (compare \cite{Comet2012,Leloup2003}), consisting of the two species \textit{Per} and \textit{Cry} (denoted by $\Per$ and $\Cry$) forming a dual negative feedback loop. Subscripts $m$ correspond to \textit{Per} and \textit{Cry} mRNA levels, and subscripts $1,\ldots,N$, $N\in\mathbb{N}$, to different post-translational states, localizations and similar of the respective proteins. For simplicity, formation of the PER-CRY complex and (indirect) inhibition of \textit{Per} and \textit{Cry} mRNA production by this complex is lumped into a single inhibitory Hill equation, utilizing the separation of timescales as compared to the other relevant reactions. Daylight exposure leads to an increase in Per mRNA expression (but not of \textit{Cry} mRNA), represented by the input $u$. 
The output $y$ was chosen purely hypothetical and merely to demonstrate that an appropriate output rendering the system scale-invariant (see below) can be approximated by simple Michaelis-Menten kinetics. 

With $\hat{p}=\log\left(\Per_m^{\frac{2n+1}{n+1}}\right)+\alpha\tau_{p}(\delta_z(x))$, $\tau_{p}(\delta_z(x))=\log\left(\Per_m^{\frac{n}{n+1}}\Cry_N\right)$, $z_{P,k}=\Per_m^{-1}\Per_k$,
$z_{C,m}=\Per_m^{\frac{n}{n+1}}\Cry_m$, and $z_{C,k}=\Per_m^{\frac{n}{n+1}}\Cry_k$, the system is transformed into the normal form
\begin{align*}
\frac{d}{dt}\hat{p}				&= \frac{2n+1+\alpha n}{n+1}\Omega(\hat{u})+\alpha\left(k_p\frac{z_{C,N-1}}{z_{C,N}}-k_{DC}\right)\\
\hat{u}										&= \exp(-\hat{p})u\\
\frac{d}{dt}z_{P,1}				&= k_{TL}-z_{P,1}(k_p+k_{DP}+\Omega(\hat{u}))\\
\frac{d}{dt}z_{P,k}				&= z_{P,k-1}-z_{P,k}(k_p+k_{DP}+\Omega(\hat{u}))\\
\frac{d}{dt}z_{P,N} 			&= k_pz_{P,N-1}-z_{P,N}(k_{DP}+\Omega(\hat{u}))\\
\frac{d}{dt}z_{C,m}				&= v_C\frac{1}{z_{C,N}^n z_{P,N}^n}-z_{C,m}\left(k_{Dm}-\frac{n}{n+1}\Omega(\hat{u})\right)\\
\frac{d}{dt}z_{C,1} 			&= k_{TL}z_{C,m}-z_{C,1}\left(k_p+k_{DC}-\frac{n}{n+1}\Omega(\hat{u})\right)\\
\frac{d}{dt}z_{C,k} 			&= k_pz_{C,k-1}-z_{C,k}\left(k_p+k_{DC}-\frac{n}{n+1}\Omega(\hat{u})\right)\\
\frac{d}{dt}z_{C,N} 			&= k_pz_{C,N-1}-z_{C,N}\left(k_{DC}-\frac{n}{n+1}\Omega(\hat{u})\right),
\end{align*}
with
\begin{align*}
\Omega(\hat{u})						&= v_p\frac{1}{z_{P,N}^nz_{C,N}^n}\pi_{\alpha\log(z_{C,N})}(\hat{u})-k_{Dm}.
\end{align*}
\begin{figure}[tb]
	\centering
	\includegraphics[width=0.4\textwidth]{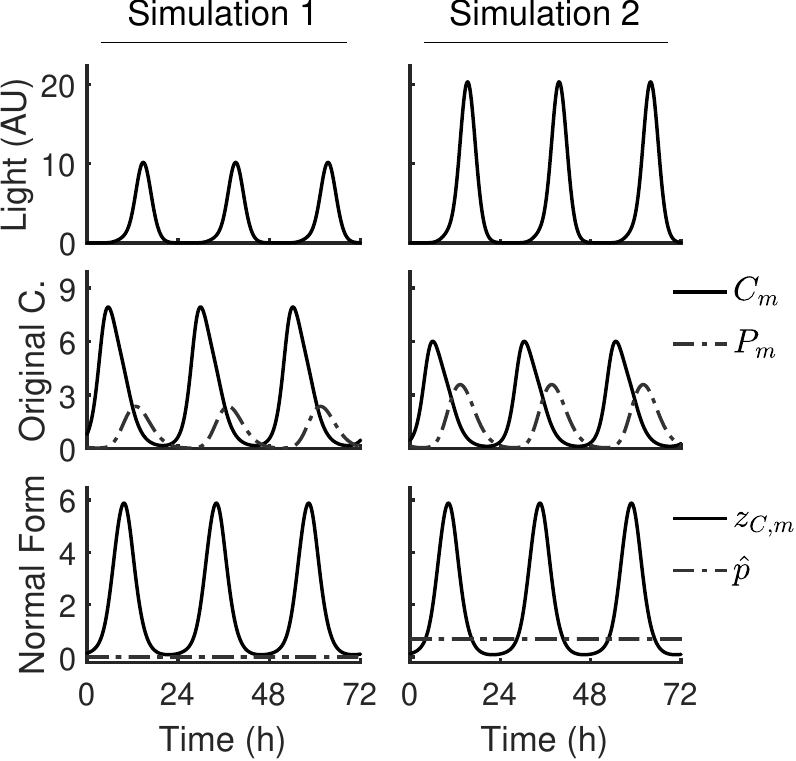}
	\caption{\label{fig:osciDynamics}
	Dynamics of the circadian rhythm model in response to the invariant input (left) and two times the invariant input (right). While the mRNA species $C_m$ and $P_m$ show oscillations with different amplitudes for the two inputs, the species $z_{c_m}$ of the normal form oscillates with the same amplitude for both inputs, while $\hat{p}$ converges to $0$ and $\log(2)$, respectively, i.e. the logarithm of the light amplitude scaling.
	}
\end{figure}

For this transformation, we utilize the non-uniqueness of the transformation into normal form to ``shape'' the transmissible input by changing the parameter $\alpha$ which multiplies $\tau_{p}(\delta_z(x))$. The simulations in Fig.~\ref{fig:osciDynamics} for example show the system excited by the transmissible input corresponding to $\alpha=-1.5$ and scaled by the factors $1$ and $2$, respectively. The dynamics of the two inputs thereby resemble the light intensities of the day-night rhythm during spring/autumn on cloudy, respectively sunny days.
\section{Discussion}
We derived a (necessary and sufficient) normal form for invariant systems demonstrating that all possible ways to achieve invariance are equivalent to a relatively simple integral-feedback mechanism. That one and only one such mechanism exists represents the most surprising result of our study, given that invariant networks are able to show a wide variety of dynamics.

Besides being of theoretical interest itself, the knowledge of the normal form lies the foundation for the rational design of invariant systems and thus for the application of the concept in engineering. Indeed, we enjoyed some distraction from the global pandemic by designing a wide variety of systems with specific properties by simply starting directly from the normal form representation itself. For example, we constructed scale-invariant systems with ramps as transmissible inputs, a rather interesting combination.
Furthermore, due to its sufficiency, the normal form also represents an excellent tool for the construction of counter-examples for various conjectures posed about invariant systems, a task we have to leave to the interested reader due to space limitations.

Finally, our last example represents a semi-realistic model of the core of the mammalian circadian rhythm. While the model is comparatively simple, it is structurally still sufficiently related to more realistic ones \cite{Comet2012,Leloup2003} such that it becomes at least plausible that the dual negative feedback structure of the mammalian circadian rhythm evolved to not only learn the phase, but also the amplitude of the day/night cycle. If this is however really the case remains a task for future theoretical and experimental studies.
\bibliographystyle{plain}  
\bibliography{mechanism_scale_invariance}

\begin{thebibliography}{10}

\bibitem{Adler2018}
M~Adler and U~Alon.
\newblock Fold-change detection in biological systems.
\newblock {\em Current Opinion in Systems Biology}, 8:81--89, 2018.

\bibitem{Bluman1989}
GW~Bluman and S~Kumei.
\newblock {\em Symmetries and differential equations}, volume~81 of {\em
  Applied mathematical sciences}.
\newblock Springer, New York, NY, 1989.

\bibitem{Comet2012}
J-P Comet, G~Bernot, A~Das, F~Diener, C~Massot, and A~Cessieux.
\newblock Simplified models for the mammalian circadian clock.
\newblock {\em Procedia Computer Science}, 11:127--138, 2012.

\bibitem{Goentoro2009}
L~Goentoro, O~Shoval, MW~Kirschner, and U~Alon.
\newblock The incoherent feedforward loop can provide fold-change detection in
  gene regulation.
\newblock {\em Molecular Cell}, 36(5):894--899, 2009.

\bibitem{Isidori1995}
A~Isidori.
\newblock {\em Nonlinear control systems}.
\newblock Springer, London, UK, third edition, 1995.

\bibitem{Lang2016}
M~Lang and E~Sontag.
\newblock Scale-invariant systems realize nonlinear differential operators.
\newblock In {\em American Control Conference (ACC)}. IEEE, 2016.

\bibitem{Lang2017}
M~Lang and E~Sontag.
\newblock Zeros of nonlinear systems with input invariances.
\newblock {\em Automatica}, 81:46--55, 2017.

\bibitem{Lazova2011}
MD~Lazova, T~Ahmed, D~Bellomo, R~Stocker, and TS~Shimizu.
\newblock Response rescaling in bacterial chemotaxis.
\newblock {\em Proceedings of the National Academy of Sciences},
  108(33):13870--13875, 2011.

\bibitem{Leloup2003}
J-C Leloup and A~Goldbeter.
\newblock Toward a detailed computational model for the mammalian circadian
  clock.
\newblock {\em Proceedings of the National Academy of Sciences},
  100(12):7051--7056, 2003.

\bibitem{Shoval2011}
O~Shoval, U~Alon, and E~Sontag.
\newblock Symmetry invariance for adapting biological systems.
\newblock {\em SIAM Journal on Applied Dynamical Systems}, 10(3):857--886,
  2011.

\bibitem{Shoval2010}
O~Shoval, L~Goentoro, Y~Hart, A~Mayo, E~Sontag, and U~Alon.
\newblock Fold-change detection and scalar symmetry of sensory input fields.
\newblock {\em Proceedings of the National Academy of Sciences},
  107(36):15995--16000, 2010.

\bibitem{Sontag2003}
E~Sontag.
\newblock Adaptation and regulation with signal detection implies internal
  model.
\newblock {\em Systems \& Control Letters}, 50(2):119--126, 2003.

\end{thebibliography}

\end{document}